\documentclass[12pt]{amsart}
\usepackage{geometry}
\usepackage{hyperref}
\usepackage{mathpazo}
\usepackage{amssymb,url,setspace,xcolor,amsmath}
\newtheorem{theorem}{Theorem}
\newtheorem*{theorem*}{Theorem}
\newtheorem{proposition}{Proposition}
\newtheorem{lemma}{Lemma}
\newtheorem{corollary}{Corollary}
\theoremstyle{remark}

\title[Hilbert-Schmidt Hankel operators]{Hilbert-Schmidt Hankel operators with harmonic symbols on the Bergman space of strongly pseudoconvex domains in $\mathbb{C}^n$}

\author{Timothy G. Clos}
\address[Timothy G. Clos]{Kent State University, Department of Mathematical 
	Sciences, Kent, OH 44242, USA}
\email{tclos@kent.edu}

\date{\today}
\begin{document}

\begin{abstract}
    We characterize Hilbert-Schmidt Hankel operators on the Bergman spaces of smooth bounded strongly pseudoconvex domains in $\mathbb{C}^n$ for $n \geq 2$.  We consider harmonic symbols of class $C^3$ up to the closure of the domain and show $H_{\phi}$ is Hilbert-Schmidt if and only if $\phi$ is holomorphic on the domain.
\end{abstract}

\maketitle

\section{Introduction}
Hilbert-Schmidt Hankel operators on Bergman spaces of domains in $\mathbb{C}^n$ for $n\geq 2$ are often classically studied on the unit ball and bounded pseudoconvex complete Reinhardt domains.  For a characterization of Hilbert-Schmidt Hankel operators on the Bergman space of the unit ball, see \cite{ZhuPaper}.  A characterization for the Bergman space of bounded pseudoconvex complete Reinhardt domains is seen in \cite{cz2}.  The proofs of the results in \cite{cz2} use the fact that normalized holomorphic monomials are an orthonormal basis for the Bergman space.   Additionally, the symbols studied are usually conjugate holomorphic.  Consequently, the proofs rely on explicit computations.  It is natural to study bounded strongly pseudoconvex domains as a generalization for a few reasons, including the nice diagonal asymptotics for the Bergman kernel near boundary points.  Additionally, we can use the nice mapping properties of the Berezin transform on such domains.  The advantage of this approach is one does not rely on explicit computations.  Thus we are able to generalize the results on the unit ball to other domains where one does not necessarily have a nice explicit orthonormal basis for the Bergman space.   \par 
We will briefly survey some previous work.  In one complex dimension, \cite{afp} showed that
for holomorphic function on the unit disc $\phi$, $H_{\overline{\phi}}$ is of Shatten $p$-class if and only if $\phi$ is in the holomorphic Besov $p$-space given by
\[\int_{\mathbb{D}}|\phi'(z)|^p (1-|z|^2)^{p-2}dA(z)<\infty.\]  See also \cite[Theorem 8.29]{ZhuBook}.  For conjugate holomorphic symbols on smoothly bounded strongly pseudoconvex domains, \cite{hui} characterized the Schatten class membership of the corresponding Hankel operators on Bergman spaces.  In higher complex dimensions, \cite{gs} studied Schatten-$p$ class membership of Hankel operators with conjugate holomorphic symbols that are smooth up the the closure of smooth bounded pseudoconvex domains.  In \cite{cz1}, the authors characterize Hilbert-Schmidt Hankel operators with conjugate holomorphic symbols on complex ellipsoids.  The papers \cite{tl} and later \cite{cz2} consider conjugate holomorphic symbols on various bounded Reinhardt domains. We take a slightly different approach and consider more general symbols. We consider symbols that are sufficiently smooth up to the closure of the domain and harmonic on the interior.  The domains we consider are $C^{\infty}$-smooth, bounded, and strongly pseudoconvex.   Recall symbol $\phi$ is said to be harmonic if $\phi$ has the local mean value property on $\Omega$.  For $\phi\in C^2(\overline{\Omega})$, this condition is equivalent to $\Delta\phi=0$ on $\Omega$.  Here, $\Delta$ is the (real) Laplacian.  Recall $A^2(\Omega)$ is the Hilbert space consisting of holomorphic, square integrable functions on $\Omega$.  It is well known that the point evaluation map is bounded on $A^2(\Omega)$.  So $A^2(\Omega)$ is a reproducing kernel space with kernel $K_z$.  This kernel is known as the Bergman kernel.  One can show that 
$K_z(w)$ is holomorphic in $w$ and conjugate holomorphic in $z$.  Furthermore,
\[\|K_z\|^2=K_z(z).\]  We denote $k_z$ as the normalized Bergman kernel.
That is,
\[k_z(w)=\frac{K_z(w)}{\|K_z\|}=\frac{K_z(w)}{\sqrt{K_z(z)}}.\]
A useful tool often used in the study operator theory in Bergman spaces, called the Berezin transform, was alluded to earlier in the introduction.  The Berezin transform of $f\in L^{\infty}(\Omega)$, denoted by $B(f)$, is defined as
\[B(f)(\zeta)=\left\langle fk_{\zeta},k_{\zeta}\right\rangle\] where $\zeta\in \Omega$.  We also define the distance to the boundary of $\Omega$ as
\[d_{b\Omega}(z)=\inf_{w\in b\Omega}\|z-w\|_E\] where $\|\cdot\|_E$ is the Euclidean norm.  Denoting $P^{\Omega}$ as the Bergman projection, we recall the Hankel operator with symbol $\phi\in L^{\infty}(\Omega)$ acting on the Bergman space $A^2(\Omega)$ is defined as \[H_{\phi}(f)=(I-P^{\Omega})(\phi f).\]   
We are now ready to state the main results.

\section{The Main Result}

\begin{theorem}\label{prelim}
	Let $\Omega\subset \mathbb{C}^n$ be a $C^{\infty}$-smooth bounded strongly pseuodoconvex domain and $n\geq 2$.  Suppose $\phi\in C^3(\overline{\Omega})$ and $H_{\phi}$ is Hilbert-Schmidt on $A^2(\Omega)$.  Then $\overline{\partial}_b(\phi)=0$ on $b\Omega$.  That is, the tangential component of $\overline{\partial}\phi$ vanishes on $b\Omega$.
\end{theorem}

We state the main result as a corollary to Theorem \ref{prelim}.

\begin{corollary}\label{cor1}
	Let $\Omega\subset \mathbb{C}^n$ be a $C^{\infty}$-smooth bounded strongly pseuodoconvex domain and $n\geq 2$.  Suppose $\phi\in C^3(\overline{\Omega})$ and $\phi$ is harmonic on $\Omega$.  Then $H_{\phi}$ is Hilbert-Schmidt on 
	$A^2(\Omega)$ if and only if $\phi$ is holomorphic on $\Omega$.
	
\end{corollary}
	
As seen in \cite{afp}, this corollary is not true for bounded domains in $\mathbb{C}$.  That is, on the unit disc in $\mathbb{C}$ we let $\phi(z)=\overline{z}$.  Since $\phi\in C^{\infty}(\overline{\mathbb{D}})$ and $\phi$ is conjugate holomorphic, $\overline{\phi}$ satisfies the estimate
\[\int_{\mathbb{D}}|\overline{\phi}'(z)|^2 dA(z)<\infty.\]  This implies $H_{\overline{z}}$ is Hilbert-Schmidt but $\overline{z}$ is harmonic and not holomorphic.

\section{Proof of Theorem \ref{prelim}}

We begin with the following proposition.

\begin{proposition}\label{prop1}
	Suppose $\Omega$ is a $C^{\infty}$-smooth bounded pseudoconvex domain and $\phi\in C^3(\overline{\Omega})$.  Then there exists
	$\psi\in C^2(\overline{\Omega})$ so that $\overline{\partial}(\psi)\in \textbf{dom}(\overline{\partial}^*)$ and $\phi=\psi$ on $b\Omega$.  
\end{proposition}

\begin{proof}
	Let $\rho$ be a defining function for $\Omega$ with gradient scaled so that
	$|\nabla(\rho)|=1$ on $b\Omega$.  Let us define 
	\[\psi=\phi-\rho \nu\]
	where $\nu$ is to be determined by using the compatibility condition for the domain of $\overline{\partial}^*$, called $\textbf{dom}(\overline{\partial}^*)$ as seen in \cite{ChenShawBook}.  Recall the compatibility condition for a $(0,1)$-form 
	$F=\sum_{j=1}^n F_j d\overline{z_j}$.  It states $F\in \textbf{dom}(\overline{\partial^*})$ if and only if
	\[\sum_{j=1}^n F_j \frac{\partial \rho}{\partial z_j}=0\] on $b\Omega$.  In other words, the complex normal direction of $F$ vanishes on the boundary.  Let us use this condition to solve for $\nu$.  We have \[\overline{\partial}\psi=\sum_{j=1}^n \frac{\partial \psi}{\partial \overline{z_j}}d\overline{z_j}\]
	Then the compatibility condition applied to $\overline{\partial}(\psi)$ states that on $b\Omega$,
	\[0=\sum_{j=1}^n \frac{\partial \rho}{\partial z_j}\left(\frac{\partial \phi}{\partial \overline{z_j}}-\rho\frac{\partial \nu}{\partial \overline{z_j}}-\nu \frac{\partial \rho}{\partial \overline{z_j}}\right)= \sum_{j=1}^n \frac{\partial \rho}{\partial z_j}\left(\frac{\partial \phi}{\partial \overline{z_j}}-\nu \frac{\partial \rho}{\partial \overline{z_j}}\right).  \]  Therefore,
	solving for $\nu$, we have
	\[\nu=\sum_{j=1}^n \frac{\partial \rho}{\partial z_j}\frac{\partial \phi}{\partial \overline{z_j}}. \] 	
	Therefore,
	\[\psi=\phi-\rho\sum_{j=1}^n \frac{\partial \rho}{\partial z_j}\frac{\partial \phi}{\partial \overline{z_j}} \] satisfies the proposition.

\end{proof}

It is well known that the following integral estimate holds for Hilbert-Schmidt operator $H_{\phi}$.  

\begin{lemma}\label{lem1}
	Let $\Omega$ be a bounded domain in $\mathbb{C}^n$ for $n\geq 1$ and suppose $\phi\in L^{\infty}(\Omega)$.  Then $H_{\phi}$ is Hilbert-Schmidt on $A^2(\Omega)$ if and only if
	\[\int_{\Omega}\|H_{\phi}K_z\|_{L^2(\Omega)}^2 dV(z)<\infty . \]
	
\end{lemma}

\begin{proof} The proof of this lemma can be obtained by slightly modifying the proof of \cite[Theorem 6.4]{ZhuBook} applied to the operator $H_{\phi}^*H_{\phi}$.
\end{proof}
Now we can prove Theorem \ref{prelim}.

\begin{proof}[Proof of Theorem \ref{prelim}]
It is well known that $C^{\infty}$-smoothly bounded strongly pseudoconvex domains are BC-regular (see \cite{cs21}).  That is, $B(f)$ extends to a continuous function on $\overline{\Omega}$ for any $f\in C(\overline{\Omega})$.  Without loss of generality, we also denote this extension as $B(f)$.  Furthermore, $B(f)=f$ on the boundary of the domain for any $f\in C(\overline{\Omega})$.
That is, for $z\in b\Omega$ and $\psi$ defined as in Proposition \ref{prop1}, \begin{equation}\label{eq:1}\sum_{j=1}^n \left|\frac{\partial \phi}{\partial \overline{z_j}}(z)\right|^2-\left|\sum_{j=1}^n \frac{\partial \phi}{\partial \overline{z_j}}(z)\frac{\partial \rho}{\partial z_j}(z)\right|^2=\lim_{\Omega\ni\zeta\to z}\left\langle \overline{\partial}(\phi)k_{\zeta}, \overline{\partial}(\psi)k_{\zeta}\right\rangle.\end{equation}  Let us assume that $H_{\phi}$ is Hilbert-Schmidt.  For the sake of obtaining a contradiction, suppose there exists $z_0\in b\Omega$ so that
\[\sum_{j=1}^n \left|\frac{\partial \phi}{\partial \overline{z_j}}(z_0)\right|^2-\left|\sum_{j=1}^n \frac{\partial \phi}{\partial \overline{z_j}}(z_0)\frac{\partial \rho}{\partial z_j}(z_0)\right|^2\neq 0.\]  By continuity,
\[\sum_{j=1}^n \left|\frac{\partial \phi}{\partial \overline{z_j}}(z)\right|^2-\left|\sum_{j=1}^n \frac{\partial \phi}{\partial \overline{z_j}}(z)\frac{\partial \rho}{\partial z_j}(z)\right|^2\neq 0\] for $z$ in some open patch $P$ of $b\Omega$.  
For $\varepsilon>0$, let us define $\Omega_{\varepsilon}=\{z\in \Omega: d_{b\Omega}(z)>\varepsilon\}$.  Since $\Omega$ has a $C^2$-smooth boundary, for every $p\in P$ there exists an inward pointing unit normal vector 
based at $p$, called $v_p$.  Let $V_p$ denote the union of all points along $v_p$. Now define \[U_{\varepsilon}=\left(\bigcup_{p\in P}V_p\right)\cap (\Omega\setminus \overline{\Omega_{\varepsilon}}).\]
Notice that $\text{vol}(\Omega\setminus \overline{\Omega_{\varepsilon}})\approx\varepsilon$.  By partitioning $b\Omega\setminus P$ into finitely many open sets, one can cover  
$\Omega\setminus \overline{\Omega_{\varepsilon}}$ with a finite number of open sets with measure $\text{vol}(U_{\varepsilon})$.  Moreover, since the number of sets depends on $P$, one can choose the number of covering sets to be independent of $\varepsilon$.  That is,
$\text{vol}(U_{\varepsilon})\approx \varepsilon$.
Now we use \eqref{eq:1} and the fact that $C^{\infty}$-smoothly bounded strongly pseudoconvex domains are BC-regular (see \cite{cs21}) 
to write
\[\left|\left\langle \overline{\partial}(\phi)k_z, \overline{\partial}(\psi)k_z\right\rangle\right|\geq \frac{M}{2}\] 
for all $z\in U_{\varepsilon}$ and for $\varepsilon>0$ sufficiently small and fixed $M$.  Thus we have for all $\varepsilon>0$ small there exists a subdomain $U_{\varepsilon}\subset \Omega$ so that the following hold.
\begin{enumerate}
	\item \[\left|\left\langle \overline{\partial}(\phi)k_z, \overline{\partial}(\psi)k_z\right\rangle\right|\geq \frac{M}{2}>0\] for $z\in {U_{\varepsilon}}$ and fixed $M$.
	\item The Lebesgue volume measure $\text{vol}(U_{\varepsilon})\approx \varepsilon$.
	\item $U_{\varepsilon}\subset \{w\in \Omega: d_{b\Omega}(w)<\varepsilon\}$.
\end{enumerate}
Now by \cite[Lemma 4.3.2 (ii)]{ChenShawBook}, smooth compactly supported $(0,1)$-forms are dense in the 
graph norm of $\overline{\partial}^*$.  By Proposition \ref{lem1}, $\overline{\partial}(\psi )\in \text{dom}(\overline{\partial}^*)$. That is, there exists $\{\chi_{\varepsilon}\}_{\varepsilon>0}\subset C^{\infty}_{(0,1)}(\Omega)$ and compactly supported in $\Omega$ so that $\chi_{\varepsilon}\to \overline{\partial}(\psi )$ in $L^2_{(0,1)}(\Omega)$ and 
$\overline{\partial}^*(\chi_{\varepsilon})\to \overline{\partial}^*\overline{\partial}(\psi )$ in $L^2(\Omega)$ as $\varepsilon\to 0^+$.  One can construct $\chi_{\varepsilon}$ as follows.  Let $C_1'>0$ be a fixed constant to be determined later.  Restrict $\overline{\partial}(\psi)$ componentwise on $\overline{\Omega_{C_1'\varepsilon}}$ and extend this restriction to be identically $0$ on $\mathbb{C}^n\setminus \overline{\Omega_{C_1'\varepsilon}}$.  Denote this constructed form as $\widetilde{\overline{\partial}(\psi)}_{C_1'\varepsilon}$.  Let $\{\Gamma_{C_1\varepsilon}\}_{\varepsilon>0}$ be the standard $C^{\infty}$-smooth mollifier compactly supported on $C_1\varepsilon \mathbb{B}_n$.  Now define the componentwise convolution as
\[\chi_{\varepsilon}=\widetilde{\overline{\partial}(\psi)}_{C_1'\varepsilon}\ast \Gamma_{C_1\varepsilon}. \]
Once $C_1'$ is determined, we will let $C_1=\frac{C_1'}{3}$.  This will guarantee that
\[\overline{\Omega_{C_1'\varepsilon}+C_1\varepsilon \mathbb{B}_n}\subset \Omega\] and imply that $\chi_{\varepsilon}$ will have compact support in $\Omega$.  Additionally, one can write
\[\chi_{\varepsilon}=\sum_{k=1}^n \chi_{k,\varepsilon} d\overline{z_k}\] where the component functions of $\chi_{\varepsilon}$ have the form
\[\chi_{k,\varepsilon}=\widetilde{\overline{\partial}(\psi)}_{k,C_1'\varepsilon}\ast \Gamma_{C_1\varepsilon}. \]  Here, one can also write the $(0,1)$-form 
\[\widetilde{\overline{\partial}(\psi)}_{C_1'\varepsilon}=\sum_{k=1}^n \widetilde{\overline{\partial}(\psi)}_{k,C_1'\varepsilon}d\overline{z_k}.\]  
Furthermore, since $\chi_{\varepsilon}$ is constructed by convolving functions with uniformly bounded variation with a smooth compactly supported mollifier, one can show that
\[\|\chi_{\varepsilon}-\overline{\partial}(\psi)\|_{L^2(\Omega)}\lesssim (C_1'\varepsilon)^{\frac{1}{2}}\] for all $\varepsilon>0$.  We will prove this crucial estimate for the convenience of the reader.  
It is well known that $\Omega\subset \mathbb{C}^n$ is isometrically isomorphic to a domain
$\widetilde{\Omega}\subset \mathbb{R}^{2n}$.  Likewise, we may assume that
$\widetilde{\Omega_{C_1'\varepsilon}}$ are subdomains of $\widetilde{\Omega}$ that are isometrically isomorphic to 
$\Omega_{C_1'\varepsilon}$.

\[\widetilde{\overline{\partial}(\psi)}_{k,C_1'\varepsilon} \] are smooth on the support $\overline{\Omega_{C_1'\varepsilon}}$.  Furthermore, by construction, they extend to a smooth function on $\overline{\Omega}$.  Let us assume \[\widetilde{\overline{\partial}(\psi)}_{k,C_1'\varepsilon}=\Re\left(\widetilde{\overline{\partial}(\psi)}_{k,C_1'\varepsilon}   \right)+i\Im\left(\widetilde{\overline{\partial}(\psi)}_{k,C_1'\varepsilon}\right). \]  We can write each complex variable $z_k=\Re(z_k)+i\Im(z_k)$ for each $k\in \{1,\ldots, n\}$.  Then we can consider the real and imaginary parts of $\widetilde{\overline{\partial}(\psi)}_{k,C_1'\varepsilon} $ to be functions in $2n$  real variables.  We let $\widetilde{\nabla}(\cdot)$ represent the real gradient.  We denote $\textbf{n}$ to be the $2n$-component real vector field of outward pointing unit normal vectors to $b\widetilde{\Omega_{C_1'\varepsilon}}$.  Then
since \[\Re\left(\widetilde{\overline{\partial}(\psi)}_{k,C_1'\varepsilon}\right)\equiv \Re\left(\frac{\partial \psi}{\partial\overline{z_k}}\right)\]
on $\overline{\widetilde{\Omega_{C_1'\varepsilon}}}$, we have
\[\left\|   \widetilde{\nabla}\left(\Re\left(\widetilde{\overline{\partial}(\psi)}_{k,C_1'\varepsilon}\right)      \right)\right\|_{L^{\infty}(\widetilde{\Omega_{C_1'\varepsilon})}}\leq m_1\] and 
\[\left\|  \Re\left(\widetilde{\overline{\partial}(\psi)}_{k,C_1'\varepsilon}     \right)\right\|_{L^{\infty}(\widetilde{\Omega_{C_1'\varepsilon})}}\leq m_2\] for some $m_1,m_2>0$ independent of $k$ and $C_1'\varepsilon$.  We also let $\text{div}(\cdot)$ denote the (real) divergence of a vector field.  We let $g:\widetilde{\Omega}\to \mathbb{R}^{2n}$ be a compactly supported test vector field with $C^1$ component functions (one can call this set of vector fields $\widetilde{C}_0(\widetilde{\Omega})$) and
$\|g\|_{L^{\infty}(\widetilde{\Omega})}\leq 1$.  Now by \cite[Proposition 9.3, ii]{Br} combined with \cite[Remark 6, pg. 269]{Br} on functions of bounded variation, we will use integration by parts to compute the following.
\begin{align*}\left|\int_{\widetilde{\Omega}} \Re\left(\widetilde{\overline{\partial}(\psi)}_{k,C_1'\varepsilon}\right)\text{div}(g)\, dV\right|
=&\left|\int_{\widetilde{\Omega_{C_1'\varepsilon}}} \Re\left(\widetilde{\overline{\partial}(\psi)}_{k,C_1'\varepsilon}\right)\text{div}(g)\, dV\right|\\
\leq& \left|\int_{\widetilde{\Omega_{C_1'\varepsilon}}} \widetilde{\nabla}\left(\Re\left(\widetilde{\overline{\partial}(\psi)}_{k,C_1'\varepsilon}\right)\right) \cdot g\, dV\right|\\
+& \left|\int_{b\widetilde{\Omega_{C_1'\varepsilon}}} \Re\left(\widetilde{\overline{\partial}(\psi)}_{k,C_1'\varepsilon}\right) g\,\cdot \textbf{n}\, d\sigma_{\varepsilon}\right|\\
\leq & \text{vol}(\widetilde{\Omega_{C_1'\varepsilon}})m_1 \|g\|_{L^{\infty}(\widetilde{\Omega})}+\sigma_{\varepsilon}(b\widetilde{\Omega_{C_1'\varepsilon}})m_2 \|g\|_{L^{\infty}(\widetilde{\Omega})}\\
\lesssim & \|g\|_{L^{\infty}(\widetilde{\Omega})}.\\ 
\end{align*} 
Here, $\sigma_{\varepsilon}$ is the surface area measure on $b\widetilde{\Omega_{C_1'\varepsilon}}$.  We additionally note that the surface area of $b\widetilde{\Omega_{C_1'\varepsilon}}$ also has a uniform upper bound independent of $C_1'\varepsilon$ and $g$.  That is, there exists $m_3>0$ independent of $k$, $g$, and $C_1'\varepsilon$ so that 
\[\left|\int_{\widetilde{\Omega}} \Re\left(\widetilde{\overline{\partial}(\psi)}_{k,C_1'\varepsilon}\right)\text{div}(g)\, dV\right|\leq m_3 \|g\|_{L^{\infty}(\widetilde{\Omega})}.\]  A similar estimate holds for the imaginary part.  Therefore, by taking supremum over all such test vector fields $g$, this estimate implies that the real and imaginary parts of  \[\widetilde{\overline{\partial}(\psi)}_{k,C_1'\varepsilon} \] have bounded variation.  In other words, the distributional derivative of both the real and imaginary parts are finite Radon measures.  We let $\text{TV}$ denote total variation.  That is, 
\[\text{TV}\left(\Re\left(\widetilde{\overline{\partial}(\psi)}_{k,C_1'\varepsilon}\right)\right)=\sup\left\{\left|\int_{\widetilde{\Omega}} \Re\left(\widetilde{\overline{\partial}(\psi)}_{k,C_1'\varepsilon}\right)\text{div}(g)\, dV\right|:g\in \widetilde{C}_0(\widetilde{\Omega})\, ,\, \|g\|_{L^{\infty}(\widetilde{\Omega})}\leq 1 \right\}.\]  Additionally, there is a uniform bound on the total variation of both the real and imaginary parts not depending on $C_1'\varepsilon$ and not dependent on $k$ because of 
\[\sup\left\{\left|\int_{\widetilde{\Omega}} \Re\left(\widetilde{\overline{\partial}(\psi)}_{k,C_1'\varepsilon}\right)\text{div}(g)\, dV\right|:g\in \widetilde{C}_0(\widetilde{\Omega})\, ,\, \|g\|_{L^{\infty}(\widetilde{\Omega})}\leq 1 \right\}\leq m_3.\]  A similar estimate holds for the imaginary part $\Im\left(\widetilde{\overline{\partial}(\psi)}_{k,C_1'\varepsilon}\right)$.  
There is also an estimate on the translation error for functions of bounded variation.  Using \cite[Lemma 14.37, pg. 483]{GL} (see also \cite[Proposition 9.3, iii]{Br} and \cite[Remark 6, pg. 269]{Br}) applied to the real and imaginary parts of compactly supported
$\widetilde{\overline{\partial}(\psi)}_{k,C_1'\varepsilon}$, we have the following translation error estimate.  Here, we are also using the fact that $\Omega$ is isometrically isomorphic to $\widetilde{\Omega}$ to conclude the following.
\begin{align*}\left\| \widetilde{\overline{\partial}(\psi)}_{k,C_1'\varepsilon}(\cdot-h)-\widetilde{\overline{\partial}(\psi)}_{k,C_1'\varepsilon}(\cdot)\right\|_{L^1(\Omega)}
\leq & \left\| \Re\left(\widetilde{\overline{\partial}(\psi)}_{k,C_1'\varepsilon}(\cdot-h)-\widetilde{\overline{\partial}(\psi)}_{k,C_1'\varepsilon}(\cdot)\right)\right\|_{L^1(\Omega)}\\
+&\left\|\Im\left(\widetilde{\overline{\partial}(\psi)}_{k,C_1'\varepsilon}(\cdot-h)-\widetilde{\overline{\partial}(\psi)}_{k,C_1'\varepsilon}(\cdot)\right)\right\|_{L^1(\Omega)} \\
\lesssim &\text{TV}\left(\Re\left(\widetilde{\overline{\partial}(\psi)}_{k,C_1'\varepsilon}\right)\right)\|h\|_E+\text{TV}\left(\Im\left(\widetilde{\overline{\partial}(\psi)}_{k,C_1'\varepsilon}\right)\right)\|h\|_E\\
\lesssim& \|h\|_E.
\end{align*} 
Now by convolving with our approximate identity $\Gamma_{C_1\varepsilon}$ and using Fubini's theorem, we get the following upper bound for all $k\in \{1,\ldots,n\}$ to arrive at the required $L^1$ estimate.

\begin{align}\label{eq0}
\left\|\chi_{k,\varepsilon}-\widetilde{\overline{\partial}(\psi)}_{k,C_1'\varepsilon}\right\|_{L^1(\Omega)}
\leq & \int_{\Omega}\int_{C_1\varepsilon \mathbb{B}_n}\left|\widetilde{\overline{\partial}(\psi)}_{k,C_1'\varepsilon}(z-h)-\widetilde{\overline{\partial}(\psi)}_{k,C_1'\varepsilon}(z)\right|\left|\Gamma_{C_1\varepsilon}(h)\right|dV(h)dV(z)\\
=&\int_{C_1\varepsilon\mathbb{B}_n}\int_{\Omega}\left|\widetilde{\overline{\partial}(\psi)}_{k,C_1'\varepsilon}(z-h)-\widetilde{\overline{\partial}(\psi)}_{k,C_1'\varepsilon}(z)\right|\left|\Gamma_{C_1\varepsilon}(h)\right|dV(z)dV(h)\\
=&\int_{C_1\varepsilon\mathbb{B}_n}\left\| \widetilde{\overline{\partial}(\psi)}_{k,C_1'\varepsilon}(\cdot-h)-\widetilde{\overline{\partial}(\psi)}_{k,C_1'\varepsilon}(\cdot)\right\|_{L^1(\Omega)}\left|\Gamma_{C_1\varepsilon}(h)\right|dV(h)\\
\lesssim & \int_{C_1\varepsilon\mathbb{B}_n}\left\| h\right\|_{E}\left|\Gamma_{C_1\varepsilon}(h)\right|dV(h)\\ 
\leq & C_1 \varepsilon \int_{C_1\varepsilon\mathbb{B}_n}\left|\Gamma_{C_1\varepsilon}(h)\right|dV(h)\\
\label{eq01}\lesssim & C_1'\varepsilon.
\end{align}
Additionally, by Young's inequality for convolutions we write  
\begin{align}\label{y1}
  \|\chi_{k,\varepsilon}\|_{L^{\infty}(\Omega)}+\left\|  \widetilde{\overline{\partial}(\psi)}_{k,C_1'\varepsilon}   \right\|_{L^{\infty}(\Omega)}
	 \leq&    \left\|\widetilde{\overline{\partial}(\psi)}_{k,C_1'\varepsilon}         \right\|_{L^{\infty}(\Omega)}\|         \Gamma_{C_1\varepsilon}\|_{L^{1}(\mathbb{C}^n)}+\left\|  \frac{\partial \psi}{\partial\overline{z_k}}  \right\|_{L^{\infty}(\Omega)}\leq \alpha_1 
\end{align}
for all $k\in \{1,\ldots, n\}$ and $\alpha_1$ independent of $C_1'\varepsilon>0$.
We also note that for each $k\in \{1,\ldots,n\}$,
\[\left\| \widetilde{\overline{\partial}(\psi)}_{k,C_1'\varepsilon}-\frac{\partial \psi}{\partial\overline{z_k}}       \right\|^2_{L^2(\Omega_{C_1'\varepsilon})}=0\] implies that
\begin{align}\label{eq3}
\left\| \widetilde{\overline{\partial}(\psi)}_{k,C_1'\varepsilon}-\frac{\partial \psi}{\partial\overline{z_k}}       \right\|^2_{L^2(\Omega)}=&\left\| \widetilde{\overline{\partial}(\psi)}_{k,C_1'\varepsilon}-\frac{\partial \psi}{\partial\overline{z_k}}  \right\|^2_{L^2(\Omega\setminus \Omega_{C_1'\varepsilon})}\\
\leq& \left(\left\|\widetilde{\overline{\partial}(\psi)}_{k,C_1'\varepsilon}\right\|_{L^{\infty}(\Omega)}+\left\|  \frac{\partial \psi}{\partial \overline{z_k}} \right\|_{L^{\infty}(\Omega)}\right)^2\,\text{vol}(\Omega\setminus \Omega_{C_1'\varepsilon})\\
\leq & \left(2\left\|  \frac{\partial \psi}{\partial \overline{z_k}} \right\|_{L^{\infty}(\Omega)}\right)^2\,\text{vol}(\Omega\setminus \Omega_{C_1'\varepsilon})\\
\label{eq4}\lesssim &\, \text{vol}(\Omega\setminus \Omega_{C_1'\varepsilon}).
\end{align}
Thus we will combine these estimates to arrive at the following estimate for $L^2$.  That is, we will bound the first term in the sum below by the $L^1$ upper bound and $L^{\infty}$ upper bound as seen on the previous page in the inequalities \eqref{eq0} to \eqref{eq01} and \eqref{y1}.  We will bound the second term in the sum below using \eqref{eq3} to \eqref{eq4} above.  We have

\begin{align*}
	\|\chi_{\varepsilon}-\overline{\partial}(\psi)\|_{L^2(\Omega)}\leq & \left\|\chi_{\varepsilon}-\widetilde{\overline{\partial}(\psi)}_{C_1'\varepsilon}\right\|_{L^2(\Omega)}+\left\|\widetilde{\overline{\partial}(\psi)}_{C_1'\varepsilon}-\overline{\partial}(\psi)\right\|_{L^2(\Omega)}\\
	=& \left(\sum_{k=1}^n \left\|\chi_{k,\varepsilon}-\widetilde{\overline{\partial}(\psi)}_{k,C_1'\varepsilon}\right\|^2_{L^2(\Omega)}\right)^{\frac{1}{2}}+\left(\sum_{k=1}^n \left\| \widetilde{\overline{\partial}(\psi)}_{k,C_1'\varepsilon}-\frac{\partial \psi}{\partial\overline{z_k}}       \right\|^2_{L^2(\Omega)}\right)^{\frac{1}{2}}      \\
	\leq & \alpha_1^{\frac{1}{2}}\left(\sum_{k=1}^n \left\|\chi_{k,\varepsilon}-\widetilde{\overline{\partial}(\psi)}_{k,C_1'\varepsilon}\right\|_{L^1(\Omega)}\right)^{\frac{1}{2}}+\left(\sum_{k=1}^n \left\| \widetilde{\overline{\partial}(\psi)}_{k,C_1'\varepsilon}-\frac{\partial \psi}{\partial\overline{z_k}}       \right\|^2_{L^2(\Omega\setminus \Omega_{C_1'\varepsilon})}\right)^{\frac{1}{2}}\\
	\leq & \left(n \alpha_1\right)^{\frac{1}{2}}\alpha_2 \left(C_1'\varepsilon\right)^{\frac{1}{2}}+\alpha_3 \sqrt{n} \left(\text{vol}(\Omega\setminus\Omega_{C_1'\varepsilon})\right)^{\frac{1}{2}}\\
	\lesssim & \left(C_1'\varepsilon\right)^{\frac{1}{2}}.
\end{align*}
Where constants $\alpha_1,\alpha_2, \alpha_3>0$ are independent of $C_1'\varepsilon$.
This proves the estimate
\[\|\chi_{\varepsilon}-\overline{\partial}(\psi)\|_{L^2(\Omega)}\lesssim (C_1'\varepsilon)^{\frac{1}{2}}.\]  
Let us now estimate the following from below.
\begin{align*}\frac{M}{2}\leq& \left|\left\langle \overline{\partial}(\phi)k_z, \overline{\partial}(\psi)k_z\right\rangle\right|
\leq \left|\left\langle \overline{\partial}(\phi)k_z, \left(\overline{\partial}(\psi)-\chi_{\varepsilon}\right)k_z\right\rangle\right|+\left|\left\langle \overline{\partial}(\phi)k_z, \chi_{\varepsilon}k_z\right\rangle\right|\\
=&\left|\left\langle \overline{\partial}(\phi)k_z, \left(\overline{\partial}(\psi)-\chi_{\varepsilon}\right)k_z\right\rangle\right|+\left|\left\langle \overline{\partial}(H_{\phi}k_z), \chi_{\varepsilon}k_z\right\rangle\right|\\
\leq & \sum_{j=1}^n \left|\left\langle \frac{\partial\phi}{\partial \overline{w_j}}\left(\overline{\frac{\partial \psi}{\partial \overline{w_j}}}-\overline{\chi_{j,\varepsilon}}\right)k_z, k_z\right\rangle\right|+\left|\left\langle H_{\phi}k_z, \overline{\partial}^*\left(\chi_{\varepsilon}k_z\right)\right\rangle\right|.
\end{align*} for all $z\in {U_{\varepsilon}}$.  Since \[\text{vol}(U_{\varepsilon})\approx \varepsilon,\] there exists $C_2>1$ independent of $C_1'$ and $\varepsilon$ so that \[\frac{1}{C_2}\varepsilon\leq \text{vol}(U_{\varepsilon})\leq C_2\varepsilon \]   Now we integrate the above string of inequalities over $U_{\varepsilon}$ and use the Cauchy-Schwarz estimate to get the following.
\begin{align*}\frac{M \varepsilon}{2C_2}\leq& \int_{U_{\varepsilon}}\left|\left\langle \overline{\partial}(\phi)k_z, \overline{\partial}(\psi)k_z\right\rangle\right|dV(z)\\
	\leq& \int_{U_{\varepsilon}}\left|\left\langle \overline{\partial}(\phi)k_z, \left(\overline{\partial}(\psi)-\chi_{\varepsilon}\right)k_z\right\rangle\right|dV(z)+\int_{U_{\varepsilon}}\left|\left\langle \overline{\partial}(\phi)k_z, \chi_{\varepsilon}k_z\right\rangle\right|dV(z)\\
	\leq & \sum_{j=1}^n \int_{U_{\varepsilon}}\left|\left\langle \frac{\partial\phi}{\partial \overline{w_j}}\left(\overline{\frac{\partial \psi}{\partial \overline{w_j}}}-\overline{\chi_{j,\varepsilon}}\right)k_z, k_z\right\rangle\right|dV(z)+\int_{U_{\varepsilon}}\left|\left\langle H_{\phi}k_z, \overline{\partial}^*\left(\chi_{\varepsilon}k_z\right)\right\rangle\right|dV(z)\\
	\leq & \sum_{j=1}^n \int_{U_{\varepsilon}}\left|\left\langle \frac{\partial\phi}{\partial \overline{w_j}}\left(\overline{\frac{\partial \psi}{\partial \overline{w_j}}}-\overline{\chi_{j,\varepsilon}}\right)k_z, k_z\right\rangle\right|dV(z)+\int_{U_{\varepsilon}}  \left\|H_{\phi}k_z\right\|  \left\|  \overline{\partial}^*\left(\chi_{\varepsilon}k_z\right)     \right\|dV(z)\\
	\leq & \left(C_2\varepsilon\right)^{\frac{1}{2}} \sum_{j=1}^n \left\|B\left(\frac{\partial\phi}{\partial \overline{w_j}}\left(\overline{\frac{\partial \psi}{\partial \overline{w_j}}}-\overline{\chi_{j,\varepsilon}}\right)\right)\right\|_{L^2(\Omega)}\\
	+&\left(\int_{U_{\varepsilon}}\|H_{\phi}k_z\|^2 dV(z)\right)^{\frac{1}{2}}\left(\int_{U_{\varepsilon}}\|\overline{\partial}^*\left(\chi_{\varepsilon}k_z\right)\|^2  dV(z)      \right)^{\frac{1}{2}}.\\
\end{align*}
It is a well known result that the Berezin transform is $L^2$-bounded on $C^{\infty}$-smooth bounded strongly pseudoconvex domains.  As a reference for the $L^2$-boundedness of the Berezin transform on such domains, we refer the reader to \cite{gs2}.  Therefore using \[\|\chi_{\varepsilon}-\overline{\partial}(\psi)\|_{L^2(\Omega)}\lesssim \left(C_1'\varepsilon\right)^{\frac{1}{2}}\] one has the estimates
\begin{align*}
\varepsilon^{\frac{1}{2}} \sum_{j=1}^n \left\|B\left(\frac{\partial\phi}{\partial \overline{w_j}}\left(\overline{\frac{\partial \psi}{\partial \overline{w_j}}}-\overline{\chi_{j,\varepsilon}}\right)\right)\right\|_{L^2(\Omega)}\lesssim &\varepsilon^{\frac{1}{2}}
\sum_{j=1}^n \left\|\left(\frac{\partial\phi}{\partial \overline{w_j}}\left(\overline{\frac{\partial \psi}{\partial \overline{w_j}}}-\overline{\chi_{j,\varepsilon}}\right)\right)\right\|_{L^2(\Omega)}\\
\lesssim & \varepsilon^{\frac{1}{2}} \sum_{j=1}^n \left\|\frac{\partial \psi}{\partial \overline{w_j}}-\chi_{j,\varepsilon}\right\|_{L^2(\Omega)}\\
\lesssim & \varepsilon^{\frac{1}{2}}\|\chi_{\varepsilon}-\overline{\partial}(\psi)\|_{L^2(\Omega)}\lesssim {(C_1')}^{\frac{1}{2}}\varepsilon.
\end{align*}
That is, there exists $\alpha_4>0$ so that
\[\left(C_2\varepsilon\right)^{\frac{1}{2}}\sum_{j=1}^n\left\|B\left(\frac{\partial\phi}{\partial \overline{w_j}}\left(\overline{\frac{\partial \psi}{\partial \overline{w_j}}}-\overline{\chi_{j,\varepsilon}}\right)\right)\right\|_{L^2(\Omega)}\leq \alpha_4 \left(C_1' \right)^{\frac{1}{2}}\varepsilon.\]  Notice that the constants $\alpha_4$ and $C_2$ are independent of $C_1'$ and $\varepsilon$.  So we can now specify $C_1'$ in our definition of $\chi_{\varepsilon}$.  We recall that $C_1=\frac{C_1'}{3}$ to ensure that $\chi_{\varepsilon}$ is compactly supported in $\Omega$.  We let
\[C_1'=\frac{M^2 }{16\, C_2^2\,\alpha_4^2}.\]  Then it is clear that

\[\left(C_2\varepsilon\right)^{\frac{1}{2}} \sum_{j=1}^n \left\|B\left(\frac{\partial\phi}{\partial \overline{w_j}}\left(\overline{\frac{\partial \psi}{\partial \overline{w_j}}}-\overline{\chi_{j,\varepsilon}}\right)\right)\right\|_{L^2(\Omega)}\leq \frac{M\varepsilon}{4C_2}. \]
It is a well known fact (see for instance \cite[Theorem 3.5.1]{hp}) that on bounded strongly pseudoconvex domains in $\mathbb{C}^n$, the asymptotic expansion for the Bergman kernel on the diagonal for $z$ near the boundary has the form

\[\frac{1}{K_z(z)}\approx d_{b\Omega}^{n+1}(z).\]
Since $H_{\phi}$ is Hilbert-Schmidt, by Lemma \ref{lem1}, 
\[\left(\int_{\Omega}\|H_{\phi}K_z\|^2dV(z)\right)^{\frac{1}{2}}<\infty.\]  Thus we have that
\[\left(\int_{U_{\varepsilon}}\|H_{\phi}k_z\|^2 dV(z)\right)^{\frac{1}{2}}\lesssim \varepsilon^{\frac{n+1}{2}}\left(\int_{U_{\varepsilon}}\|H_{\phi}K_z\|^2 dV(z)\right)^{\frac{1}{2}}.\]  Now it remains to show that 
\[\left(\int_{U_{\varepsilon}}\|\overline{\partial}^*(\chi_{\varepsilon}k_z)\|^2dV(z)\right)^{\frac{1}{2}}\lesssim \varepsilon^{-\frac{1}{2}}.\]
Using \cite[1.4.2, page 25]{hb} and Young's inequality for convolutions, one has the following bound.
\[\left\|\frac{\partial \chi_{k,\varepsilon}}{\partial \overline{w_j}}\right\|_{L^{\infty}(\Omega)}^2= \left\|   \widetilde{\overline{\partial}(\psi)}_{k,C_1'\varepsilon}\ast \frac{\partial\Gamma_{C_1\varepsilon}}{\partial \overline{w}_j}\right\|^2_{L^{\infty}(\mathbb{C}^n)}\leq \left\|\widetilde{\overline{\partial}(\psi)}_{k,C_1'\varepsilon}   \right\|_{L^{\infty}(\mathbb{C}^n)}^2 \left\|\frac{\partial\Gamma_{C_1\varepsilon}}{\partial \overline{w}_j}\right\|_{L^1(\mathbb{C}^n)}^2  \lesssim {\varepsilon}^{-2}\]
Now an application of the Morrey-Kohn-H\"{o}rmander estimate (see for example \cite[Proposition 4.3.1]{ChenShawBook}) results in the following estimate    

\[\|\overline{\partial}^*(\chi_{\varepsilon}k_z)\|_{L^2(\Omega)}^2\lesssim \sum_{j,k=1}^n \left\|\frac{\partial \chi_{k,\varepsilon}}{\partial \overline{w_j}}k_z\right\|_{L^2(\Omega)}^2\leq \sum_{j,k=1}^n \left\|\frac{\partial \chi_{k,\varepsilon}}{\partial \overline{w_j}}\right\|_{L^{\infty}(\Omega)}^2\lesssim {\varepsilon}^{-2}.  \]
We note that the general Morrey-Kohn-H\"{o}rmander estimate has another boundary integral term we did not include in this estimate.  The form $\chi_{\varepsilon}k_z$ has compact support in $\Omega$ so this boundary integral term vanishes.  Thus integrating over $z\in U_{\varepsilon}$ we have that
\[\left(\int_{U_{\varepsilon}}\|\overline{\partial}^*(\chi_{\varepsilon}k_z)\|^2dV(z)\right)^{\frac{1}{2}}\lesssim\varepsilon^{-\frac{1}{2}}.\]
Now we can combine all of these estimates to get the following.
\begin{align*}
\frac{M\varepsilon}{2C_2}\leq &\left(C_2\varepsilon\right)^{\frac{1}{2}} \sum_{j=1}^n \left\|B\left(\frac{\partial\phi}{\partial \overline{w_j}}\left(\overline{\frac{\partial \psi}{\partial \overline{w_j}}}-\overline{\chi_{j,\varepsilon}}\right)\right)\right\|_{L^2(\Omega)}\\
+&\left(\int_{U_{\varepsilon}}\|H_{\phi}k_z\|^2 dV(z)\right)^{\frac{1}{2}}\left(\int_{U_{\varepsilon}}\|\overline{\partial}^*\left(\chi_{\varepsilon}k_z\right)\|^2  dV(z)      \right)^{\frac{1}{2}}\\
\leq & \frac{M\varepsilon}{4C_2}+C'\varepsilon^{-\frac{1}{2}}\varepsilon^{\frac{n+1}{2}}\left(\int_{U_{\varepsilon}}\|H_{\phi}K_z\|^2 dV(z)\right)^{\frac{1}{2}}.
\end{align*}
Thus we have that \[\frac{M \varepsilon}{4C_2}\leq C' \varepsilon^{\frac{n}{2}}\left(\int_{U_{\varepsilon}}\|H_{\phi}K_z\|^2 dV(z)\right)^{\frac{1}{2}}\] for some $C'>0$ independent of $\varepsilon$.  This is equivalent to 
\[M'\leq  \varepsilon^{\frac{n-2}{2}}\left(\int_{U_{\varepsilon}}\|H_{\phi}K_z\|^2 dV(z)\right)^{\frac{1}{2}}\] for some $M'>0$ independent of $\varepsilon$.  We assumed that $H_{\phi}$ is Hilbert-Schmidt on $A^2(\Omega)$.  Therefore by Lemma \ref{lem1}, \[\| H_{\phi}K_z\|\in L^2(\Omega).\]  Also, the Lebesgue measure of $U_{\varepsilon}$ goes to $0$ as $\varepsilon\to 0^+$.  Therefore,
\[\lim_{\varepsilon\to 0^+}\left(\int_{U_{\varepsilon}}\|H_{\phi}K_z\|^2 dV(z)\right)^{\frac{1}{2}}=0.\]  Since we assumed that $n\geq 2$, we obtain a contradiction.  Thus we conclude that

\[\sum_{j=1}^n \left|\frac{\partial \phi}{\partial \overline{z_j}}(z)\right|^2-\left|\sum_{j=1}^n \frac{\partial \phi}{\partial \overline{z_j}}(z)\frac{\partial \rho}{\partial z_j}(z)\right|^2=0\] on $b\Omega$.  We recognize \[\sum_{j=1}^n \frac{\partial \phi}{\partial \overline{z_j}}(z)\frac{\partial \rho}{\partial z_j}(z)\] as the complex normal component of $\overline{\partial}(\phi)(z)$.  Hence the complex tangential component of $\overline{\partial}\phi$ is $0$ on $b\Omega$.  That is,
$\overline{\partial}_b(\phi)=0$.

\end{proof}

\section{Proof of Corollary \ref{cor1} }

\begin{proof}
Let us assume that $\phi$ satisfies the conditions of Theorem \ref{prelim} and is additionally harmonic on $\Omega$.
If $\phi$ is holomorphic on $\Omega$, then $H_{\phi}\equiv 0$ is trivially Hilbert-Schmidt.  Therefore, suppose $H_{\phi}$ is Hilbert-Schmidt on $A^2(\Omega)$.  Then by Theorem \ref{prelim}, $\overline{\partial}_b(\phi)=0$.
Thus the restriction of $\phi$ to $b\Omega$ is a CR-function.  Therefore, $\phi|_{b\Omega}$ extends to a holomorphic function on $\Omega$.  See \cite[Theorem 3.2.2]{ChenShawBook} for a proof of this result.  Hence $\phi$ is holomorphic by the maximum principle for harmonic functions.  

\end{proof}

\section{Acknowledgments}

I wish to thank S\"{o}nmez \c{S}ahuto\u{g}lu and Trieu Le for helpful suggestions on a preliminary version of this manuscript. I also thank the referee for their detailed comments.

\bibliographystyle{amsalpha}
\bibliography{hsrefs}

\end{document}